\documentclass[11pt,twoside,a4paper]{article}
\usepackage{amsmath,amssymb,amsthm}

\usepackage{graphicx,psfrag}
\newtheorem{theorem}{Theorem}[section]
\newtheorem{lemma}[theorem]{Lemma}
\newtheorem{proposition}[theorem]{Proposition}
\newtheorem{corollary}[theorem]{Corollary}

\theoremstyle{definition}
\newtheorem{example}[theorem]{Example}

\theoremstyle{remark}
\newtheorem{remark}[theorem]{Remark}
\newtheorem{definition}[theorem]{Definition}

\newcommand{\R}{\mathbb{R}}

\newcommand{\N}{\mathbb{N}}

\newcommand{\C}{\mathbb{C}}

\newcommand{\cA}{\mathcal{A}}

\newcommand{\cM}{\mathcal{M}}

\newcommand{\al}{\alpha}

\newcommand{\om}{\omega}
\newcommand{\Om}{\Omega}
\newcommand{\si}{\sigma}

\newcommand{\la}{\lambda}

\renewcommand{\phi}{\varphi}

\newcommand{\CAT}{\operatorname{CAT}}

\newcommand{\id}{\operatorname{id}}

\newcommand{\crt}{\operatorname{crt}}

\newcommand{\sm}{\setminus}
\newcommand{\sub}{\subset}

\begin{document}

\title{A M\"obius Characterization of Metric Spheres}
\author{Thomas Foertsch \&  Viktor Schroeder}

\maketitle

\begin{abstract}
In this paper we characterize compact extended Ptolemy metric spaces with 
many circles up to 
M\"obius equivalence. This characterization yields a M\"obius 
characterization of the $n$-dimensional spheres $S^n$
and hemispheres $S^n_+$ when endowed with their chordal metrics.
In particular,
we show that every compact extended Ptolemy metric space with
the property that every three points are contained in a circle
is M\"obius equivalent to $(S^n,d_0)$ for some $n\ge 1$, the
$n$-dimensional sphere $S^n$ with its chordal metric.
\end{abstract}

%%%%%%%%%%%%%%%%%%%%%%%%%%%%%%%%%%%%%%%%%%%%%%%%%%%%%%%%%%%%%%%%%%%%%%%%%%%%%%%%%%%%%%%%%%%%%%%%%%%%%%%%%%%%%%%%%%%%%%%%%%%%%%%%%%%%%%%%%%%%%%%%%%%%%%%%%%%%%%%%%%

\section{Introduction}

Our main theorems in this paper, Theorems  \ref{main-theorem} and 
\ref{theo-3pt-int}, characterize  
spheres and hemispheres in the context of metric 
M\"obius geometry by the existence of many circles. 
It is useful to recall in this context
the classical charcterization of circles which goes
back to Claudius Ptolemaeus (ca. 90-168). 

\begin{theorem} \label{theo-pt} (Ptolemy's Theorem) \\
Consider four points in the Euclidean space, $x_1,x_2,x_3,x_4\in \mathbb{E}^n=(\mathbb{R}^n,d)$. Then
\begin{equation} \label{eqn-ptolemy-inequality}
d(x_1,x_3) \, d(x_2,x_4) \; \le \; d(x_1,x_2) \, d(x_3,x_4) \; + \; d(x_1,x_4) \, d(x_3,x_2).
\end{equation}
Moreover, equality holds if and only if the four points lie on a circle $C$ such that $x_2$ and $x_4$
lie in different components of $C\setminus \{ x_1 ,x_3\}$.
\end{theorem}

A metric space $(X,d)$ is called a {\em Ptolemy metric space} if inequality (\ref{eqn-ptolemy-inequality})
holds for arbitrary quadruples in the space.
We call a subset
$\si \sub X$ a
{\em circle}, if
$\si$ is homeomorphic to
$S^1$ and for any for points
$x_1,x_2,x_3,x_4$ on
$\si$ (in this order) we have
equality in 
(\ref{eqn-ptolemy-inequality}).

Consider on $S^n$ the
chordal metric $d_0$,
i.e. the metric induced by its standard embedding $S^n \hookrightarrow \mathbb{E}^{n+1}$.
Via the stereographic projection
$(S^n,d_0)$ is 
M\"obius equivalent to
$\mathbb{E}^n\cup\{\infty\}$
and by Theorem \ref{theo-pt}
a Ptolemy metric space such that through
any three points there exists a circle.

\begin{theorem} \label{main-theorem}
Let 
$(X,d)$ be a compact extended Ptolemy metric space which contains
at least three points.
If any three points in $X$ lie on a circle, then
$X$ is M\"obius equivalent to 
$(S^n,d_0)$ for some $n\in \mathbb{N}$.
\end{theorem} 

For the notion of extended metric spaces, i.e. metric spaces allowing one point at infinity, see
section \ref{sec:moeb-geo}. \\

Finally, we define a Ptolemy segment to be the analogon of a  segment of a circle. 
\begin{definition}
Let $d$ be a metric on some closed, bounded interval $I\subset \mathbb{R}$, giving back its
standard topology. Then $(I,d)$ is called a {\em Ptolemy segment} if equality holds in
Inequality (\ref{eqn-ptolemy-inequality}) whenever $x_1,x_2,x_3,x_4$ lie in this order
on the segment $I$.
\end{definition}

For Ptolemy metric spaces with the property that through each three
given points there exists a Ptolemy segment containing these points,  
we obtain the analogon of Theorem \ref{main-theorem} as follows.

\begin{theorem} \label{theo-3pt-int}
Let 
$(X,d)$ be a compact extended Ptolemy metric space containing at least three points.
If any three points in $X$ lie on a circle or
on a Ptolemy segment, then
$(X,d)$ is 
M\"obius equivalent to 
either some 
$(S^n,d_0)$  or to some hemisphere 
$(S^n_+,d_0)$ for some $n\in \mathbb{N}$.
\end{theorem}

On the way to our main results we obtain some other
results about Ptolemy metric spaces.

There are many
different isometry types of Ptolemy circles (cf. Propositions \ref{prop-segment} and \ref{prop-circle}
as well as Remarks \ref{rem-segments} and \ref{rem-circles}),  but
there is only one M\"obius type.

\begin{theorem} \label{thm:moebch-circ}
Let $C$ and $C'$ be Ptolemy circles.
Let $x_1,x_2,x_3$ and $x'_1,x'_2,x'_3$ be distinct
points on $C$ respectively on $C'$.
Then there exists a unique
M\"obius homeomorphism
$\varphi:C\to C'$ with $\varphi(x_i)=x'_i$.
\end{theorem}

From this we obtain almost immediately the analogon of Theorem \ref{main-theorem},
when considering Ptolemy metric spaces with the property that through each four points there 
exists a Ptolemy circle containing these four points.

\begin{corollary} \label{cor-4points}
An extended Ptolemy metric space with the property that through each four points 
of the space there exists a Ptolemy circle containing these points, is M\"obius equivalent 
to $S^1$ when endowed with its chordal metric.
\end{corollary}

Note that here we do not have to assume the space to be compact. \\
However, if we drop the assumption of compactness in Theorem \ref{main-theorem},
the conclusion no longer holds. In fact
there are a lot such non-compact spaces with interesting properties (cf. Example 
\ref{example-non-unique-circle} and 
Theorem \ref{theo-emb-3-pt-circle-space}).

The paper is structured as follows: \\
In section \label{sec:moeb-geo}
we give a short introduction to metric 
M\"obius Geometry and state other preliminary results.
In section \label{sec-charac-circles} we 
classify Ptolemy circles and segments up to
M\"obius equivalence and up to isometry.
In section
\label{sec-many circles}
we proof the main results.

Our main motivation to study
Ptolemy metric spaces is the fact that they
occur naturally on the boundary at infinity of a 
$\CAT(-1)$ space. This relation is described in
\cite{FS1}.

\section{Metric M\"obius Geometry} \label{sec:moeb-geo}

\subsection{M\"obiusstructure}

Let
$X$ 
be a set which contains at least
two points.
An {\em extended metric} on 
$X$ 
is a map 
$d:X\times X \to [0,\infty ]$,
such that
there exists
a set
$\Om(d) \sub X$ with
cardinality
$\# \Om(d) \in \{ 0,1\}$, 
such that 
$d$
restricted to the set
${X\setminus \Om(d)}$ 
is a metric 
(taking only values in $[0,\infty)$) and such that
$d(x,\omega)=\infty$ 
for all 
$x\in X\setminus \Om(d)$, $\om \in \Om(d)$. 
Furthermore $d(\omega,\omega)=0$.\\
If
$\Om(d)$
is not empty,
we sometimes denote
$\om \in \Om(d)$ simply as
$\infty$ and call it the
(infinitely) remote point
of
$(X,d)$.
We often write also
$\{\om\}$ for the set
$\Om(d)$ and
$X_{\om}$ for the set
$X\sm \{\om\}$.

The topology considered on $(X,d)$ is the topology with the 
basis consisting of all open distance balls $B_{r}(x)$
around points in $x\in X_{\om}$ and the
complements $D^C$ of all closed distance balls $D=\overline{B}_r(x)$. \\

We call an extended metric space {\em complete}, if first every
Cauchy sequence in $X_{\om}$ converges and secondly
if the infinitely remote point $\om$ exists in case that $X_{\om}$
is unbounded.
For example the real line
$(\R,d)$,
with its standard metric is {\em not} complete (as extended metric space),
while
$(\R\cup\{\infty\},d)$ is complete.

We say that a quadruple 
$(x,y,z,w)\in X^4$ is 
{\em admissible}, if no entry occurs three or
four times in the quadruple.
We denote with
$Q\sub X^4$
the set of 
admissible quadruples.
We define the {\em cross ratio triple} as the map
$\crt:\ Q \to \Sigma \subset \R P^2$ 
which maps 
admissible quadruples to points in the real projective plane defined by
$$\crt(x,y,z,w)=(d(x,y)d(z,w): d(x,z)d(y,w) : d(x,w)d(y,z)),$$
here 
$\Sigma$ 
is the subset of points 
$(a:b:c) \in \R P^2$, 
where
all entries 
$a,b,c$ 
are nonnegative or all entries are nonpositive.
Note that 
$\Sigma$ can be identified with the standard $2$-simplex,
$\{(a,b,c)\, |\, a,b,c \ge 0,\, a+b+c=1\}$.

We use the standard conventions for the calculation with 
$\infty$.
If 
$\infty$ occurs once in 
$Q$, say 
$w=\infty$,
then
$\crt(x,y,z,\infty)=(d(x,y):d(x,z):d(y,z))$.
If 
$\infty$ 
occurs twice , say $z=w=\infty$ then
$\crt(x,y,\infty,\infty)=(0:1:1)$.

Similar as for the classical cross ratio there are six possible definitions by
permuting the entries and we choose the above one. 

It is not difficult to check that 
$\crt:Q\to \Sigma$ 
is continuous,
where $Q$ and $\Sigma$ carry the obvious topologies induced by
$X$ and $\R P^2$. Thus, if $(x_i,y_i,z_i,w_i) \in Q$ for $i \in \N$ and asume
$x_i\to x,\ldots,w_i\to w$, where $(x,y,z,w)\in Q$ then 
$\crt(x_i,y_i,z_i,w_i)\to\crt(x,y,z,w)$.

A map
$f:X\to Y$
between two extended metric spaces
is called
{\em M\"obius}, if 
$f$ is injective and for all
admissible quadruples
$(x,y,z,w)$ of
$X$,
$$\crt(f(x),f(y),f(z),f(w))=\crt(x,y,z,w).$$
M\"obius maps are continuous.

Two extended metric spaces
$(X,d)$ and
$(Y,d')$ are
{\em M\"obius equivalent},
if there exists a bijective
M\"obius map
$f:X\to Y$.
In this case also
$f^{-1}$ is a M\"obius map and
$f$ is in particular a
homeomorphism.\\
We say that two extended metrics
$d$ and $d'$
on a set 
$X$ are
{\em M\"obius equivalent},
if the identity map
$\id:(X,d)\to (X,d')$
is a M\"obius map.
M\"obius equivalent metrics define
the same topology on
$X$.

A {\em M\"obius structure} on a set
$X$ is a nonempty set
$\cM$ of extended metrics on
$X$,
which are
pairwise M\"obius equivalent and which is
maximal with respect to that property.

A M\"obiusstructure defines a topology
on $X$.
In general two metrics in
$\cM$ can look very different.
However if two metrics have the same remote
point at infinity, then they are homothetic.
Since this result is crucial for our considerations, we 
state it as a Lemma. 

\begin{lemma}\label{lem:homothety}
Let $\cM$ be a M\"obiusstructure on a set
$X$, and let
$d,d'\in \cM$, such that 
$\om \in X$ is the remote point of
$d$ and of $d'$. Then there exists $\la >0$,
such that
$d'(x,y)=\la d(x,y)$ for all 
$x,y \in X$.
\end{lemma}

\begin{proof}
Since otherwise the result is trivial, we can assume that
there are distinct points
$x,y \in X\sm\{\om\}$.
Choose 
$\la >0$ such that
$d'(x,y)=\la d(x,y)$.
If $z\in X\sm\{\om\}$, then
$\crt(x,y,z,\om)$ is the same in the metric
$d$ and $d'$, hence
$(d'(x,y):d'(x,z):d'(y,z))=(d(x,y):d(x,z):d(y,z)).$
Since $d'(x,y)=\la d(x,y)$
we therefore obtain
$d'(x,z)=\la d(x,z)$ and
$d'(y,z)=\la d(y,x)$.
\end{proof}

\subsection{Ptolemy spaces}

An extended metric space $(X,d)$ is called a {\it Ptolemy space}, if for all 
quadruples of points $\{ x,y,z,w\} \in X^4$
the {\it Ptolemy inequality} holds
\begin{displaymath}
d(x,y) \, d(z,w) \; \le \; d(x,z) \, d(y,w) \; 
+ \; d(x,w) \, d(y,z)
\end{displaymath}

We can reformulate this condition in 
terms of the cross ratio triple. 
Let 
$\Delta \subset \Sigma$ 
be the set of points 
$(a:b:c)\in \Sigma$, such that
the entries 
$a,b,c$ 
satisfy the triangle inequality. This is obviously well defined.
If we identify 
$\Sigma\sub \R P^2$ 
with the standard $2$-simplex, i.e. the convex hull
of the unit vectors 
$e_1,e_2,e_3$, then 
$\Delta$ is the convex subset
spanned by 
$(0,\frac{1}{2},\frac{1}{2})$,
$(\frac{1}{2},0,\frac{1}{2})$ and
$(\frac{1}{2},\frac{1}{2},0)$. 
We denote by
$\hat{e}_1:=(0:1:1)$,
$\hat{e}_2:=(1:0:1)$ and
$\hat{e}_3:=(1:1:0)$.
Note that also 
$\Delta$ is homeomorphic to a $2$-simplex and
$\partial \Delta$ is homeomorphic to $S^1$.
Then an
extended space is  Ptolemy, if 
$\crt(x,y,z,w) \in \Delta$ 
for all allowed quadruples $Q$.

This description shows that 
the Ptolemy property is M\"obius invariant and thus a property
of the M\"obiusstructure
$\cM$.

The importance of the Ptolemy property comes from the following fact.

\begin{theorem}
A M\"obiusstructure
$\cM$ on a set
$X$ is
Ptolemy, if and only if for all
$z\in X$ there
exists 
$d_z\in\cM$ with
$\Om(d_z)=\{z\}$. 
\end{theorem}

\begin{proof}
Assume that 
$\cM$ is Ptolemy and that
$z\in X$.
Choose some 
$d\in \cM$.
If $z\in \Om(d)$, we have
our desired metric.
If not we define
$d_z:X\times X\to [0,\infty]$ by
\begin{align*}
d_z(x,y)&\ \  =\ \ \frac{d(x,y)}{d(z,x)d(z,y)}\ \  &{\mbox for }\ \ \ x,y\in X\sm(\Om(d)\cup \{z\}), \\
d_z(x,\om)&\ \ =\ \ \frac{1}{d(z,x)} &{\mbox for}\ \  x\in X\sm \Om(d),\\
d_z(z,x)&\ \ =\ \ \infty &{\mbox for}\ \  x\in X\sm \{z\}
\end{align*}

Since for 
$x,y,w \in X\sm\{z\}$
\begin{align*}
&(d_z(x,y): d_z(y,w):d_z(x,w))\ \ = \\
&\ \ \ \ \ \ \  (d(x,y)\,d(z,w):d(x,z)\ d(y,w):d(x,w)\,d(y,z))\in \Delta
 \end{align*}
we see that 
$d_z$
satisfies the triangle inequality and hence
$d_z \in \cM$. 

If on the other hand for every
$z\in X$
there is a metric
$d_z\in \cM$ 
with
$\Om(d_z) =\{z\}$,
then for all
$x,y,w \in X\sm\{z\}$ and
all 
$d\in \cM$
\begin{align*}
 &(d(x,y)\,d(z,w):d(x,z)\ d(y,w):d(x,w)\,d(y,z))=\\
&(d_z(x,y)\,d_z(z,w):d_z(x,z)\ d_z(y,w):d_z(x,w)\,d_z(y,z))=\\
&(d_z(x,y): d_z(y,w):d_z(x,w))\in \Delta
\end{align*}
which implies the Ptolemy inequality.
\end{proof}

For a Ptolemy metric space $(X,d)$ we call $(X,d_z)$ the inversion of $(X,d)$ at $z\in X$. \\

On the other hand we allways have bounded metrics
in a M\"obius structure.

\begin{lemma}
Let $\cM$ be a Ptolemy M\"obius structure on a set
$X$. Then there exists a bounded metric
$d\in \cM$.
\end{lemma}
\begin{proof}
Take any $d\in \cM$.
We can assume that $(X,d)$ is unbounded.
Let $o\in X$, $o\notin \Om(d)$. On $X$ we define a new bounded metric $d_o \in \cM$.  We let 
\begin{displaymath} 
d_o(x,x') \; := \; \frac{d(x,x')}{(d(x,o)+1)(d(x',o)+1)} \hspace{1cm} \forall x,x'\in X\setminus\{\infty\}.
\end{displaymath}
Note that this expression extends continuously to the point $\infty$ (in case that $\infty \in X$ exists), with
$d_o(\infty,\infty)=0$, and $d_o(x,\infty)= 1/(d(x,o)+1)$.
Then $d_o$ defines a bounded metric on $X$ which is M\"obius equivalent to the old one. In fact, one 
can see this construction isn the following way:
extend $(X,d)$
to a (Ptolemy!) extended metric space $(X\cup \{ \omega \}, \bar{d})$ with some
additional finite point $\omega$, 
where $\bar{d}|_{X\times X}:=d$, $\bar{d}(\omega ,\omega ):=0$
and $\bar{d}(\omega ,x) := d(o,x) +1$. Then apply an involution at $\omega$. 
The restriction
of this metric to $X$ yields the bounded metric $d_o$. 
\end{proof}

\subsection{Circles in Ptolemy spaces}

A circle in a Ptolemy space 
$(X,d)$ is a subset $\si \sub X$ hoemeomorphic
to $S^1$  such that for distinct
points
$x,y,z,w\in \si$ (in this order)
\begin{equation}\label{eq:PT_eq}
d(x,z)d(y,w)=d(x,y)d(z,w)+d(x,w)d(y,z)
\end{equation}
Here the phrase "in this order" means that
$y$ and
$w$ are in different components of
$\si \sm\{x,z\}$.
We recall that the classical
Ptolemy theorem states, that
four points 
$x,y,z,w$ of the euclidean
plane lie on a circle (in this order), if and
only if their distances satisfy
the Ptolemy equality
(\ref{eq:PT_eq}).
One can reformulate this via the crossratio triple.
A subset
$\si$ homeomorphic to
$S^1$ is a circle, if and only if
for all admissible
quadruples $(x,y,z,w)$ of point in
$\si$ we have
$\crt(x,y,z,w) \in \partial \Delta$.
This shows that the definition
of a circle is 
M\"obius invariant and hence a concept
of the M\"obius structure.
Let
$\si$ be a circle
and let
$\om \in \si$ and
consider
$\si_{\om}=\si\sm\{\om\}$
in a metric with remote point
$\om$,
then 
$\crt(x,y,z,\om) \in \partial \Delta$
says that for 
$x,y,z \in \si_{\om}$
(in this order)
$d(x,y)+d(yz)=d(x,z)$,
i.e. it implies that
$\si_{\om}$
is a geodesic,
actually a complete
geodesic isometric
to
$\R$.

%%%%%%%%%%%%%%%%%%%%%%%%%%%%%%%%%%%%%%%%%%%%%%%%%%%%%%%%%%%%%%%%%%%%
%%%%%%%%%%%%%%%%%%%%%%%%%%%%%%%%%%%%%%%%%%%%%%%%%%%%%%%%%%%%%%%%%%%%
%%%%%%%%%%%%%%%%%%%%%%%%%%%%%%%%%%%%%%%%%%%%%%%%%%%%%%%%%%%%%%%%%%%%
%%%%%%%%%%%%%%%%%%%%%%%%%%%%%%%%%%%%%%%%%%%%%%%%%%%%%%%%%%%%%%%%%%%%

We note that via the stereographic projection the extended Ptolemy space
$\mathbb{E}^n\cup\{\infty\}$ is M\"obius equivalent to $(S^n,d_0)$,
where $d_0$ is the chordal metric, i.e. the restriction
of the metric on $\mathbb{E}^{n+1}$.

\subsection{Uniqueness of Ptolemy circles} \label{subsec:uofptcirc}

In this section we consider a 3-point Ptolemy circle space $X$,
i.e. we assume that through any three given points in $X$ there exists a circle.
We show that in the case that $X$ is compact,
circles are uniquely determined by three distinct points:

\begin{theorem} \label{thm:unique_circle}
Let $X$ be a compact extended 3-point Ptolemy circle space,
then through three distinct points there exists a unique circle.
\end{theorem}

We prove this result by using suitable involutions.
Let $p \in X$ be an arbitrary point, then consider the M\"obius equivalent extended metric 
$d_p$ on $X$
(in $d_p$ the point $p$ is the infinitely remote point).
Now one easily sees that a Ptolemy circle through $p$ is in the metric
$d_p$ a geodesic. 
Thus we can reduce the uniqueness of circles to the uniqueness
of geodesics.

Note that since $(X,d)$ is a 3-point Ptolemy circle space,
$(X\sm \{ p\},d_p)$ is a geodesic spaces, i.e. through $x,y \in X\sm \{p\}$ there exists
a geodesic in $(X\sm \{ p\},d_p)$.
We apply the following Theorem
on geodesic Ptolemy metric spaces from \cite{FLS}.

\begin{theorem} [Theorem 1.2 in \cite{FLS}]
A locally compact, geodesic Ptolemy metric space is uniquely geodesic.
\end{theorem}

To prove Theorem \ref{thm:unique_circle}, consider three distinct points
$x, y, z \in X$. By taking the involution $d_z$ and the above result,
we see that the segment of the Ptolemy circle between $x$ and $y$ which does not contain $z$ is uniquely determined. By symmetry of the argument this is also true for the two other segments. This implies the uniqueness.

We point out that the assumption of compactness is crucial, 
as the following theorem from \cite{FLS} shows.

\begin{theorem} [Theorem 1.1 in \cite{FLS}]
Every Ptolemy metric space $(X,d)$ admits an isometric embedding into a complete, geodesic Ptolemy space $(\hat{X},\hat{d})$.
\end{theorem}

Thus, start with a metric space $(X,d)$ consisting of four points,
where five of the six distances are equal to one and the remaining distance is 
equal to two. This space is Ptolemy and by the above can be isometrically embedded
into a geodesic Ptolemy space. But this space is not
uniquely geodesic. \\

In general, there might exist different
Ptolemy circles through three given points, as the following example shows.

\begin{example} \label{example-non-unique-circle}
The construction method of $(\hat{X},\hat{d})$ as in the theorem cited above allows to add
arbitrarily many geodesics between any two points such that all distance functions $d(p,\cdot )$ to points $p\in X$
are affine. Now taking any three points in $\hat{X}$ then for any choice of three such geodesics connecting the given points to 
each other, their concatenation yields a Ptolemy circle. 
\end{example}

This observation can also be formulated as follows.

\begin{theorem} \label{theo-emb-3-pt-circle-space}
Every Ptolemy metric space $(X,d)$ admits an isometric embedding into a $3$-point Ptolemy circle space.
\end{theorem} 

%%%%%%%%%%%%%%%%%%%%%%%%%%%%%%%%%%%%%%%%%%%%%%%%%%%%%%%%%%%%%%

\subsection{Convexity and Busemann Functions} 
\label{sec:bus_convex}

Recall that a geodesic metric space $(X,d)$ is called {\em distance convex}, if all its distance functions to points $z\in X$
\begin{displaymath}
d(z,\cdot ): \mathbb{R}_0^+, \hspace{1cm} x \; \mapsto \; d(z,x) \ \forall x\in X
\end{displaymath}
are convex, i.e., that their restriction to any geodesic segment in $(X,d)$ is convex. \\
A geodesic Ptolemy metric space is distance convex, which follows immediately from the Ptolemy inequality applied to
points $z,x,m,$ and $y$, where $m$ is a midpoint of $x$ and $y$, i.e. $d(x,m)=\frac{1}{2}d(x,y)=d(m,y)$. \\

Let $c:[0,\infty) \to X$ be a geodesic ray parameterized by arclength.
As usual we define the {\em Busemann function}
$b_c(x) = \lim_{t\to\infty}(d(x,c(t))-t)$. \\
If $X$ is a geodesic Ptolemy space, then $b_c$ is convex, 
being the limit of the convex functions $d(c(t),\cdot) -t$. \\

For more information on Busemann functions in geodesic Ptolemy spaces we refer the reader to \cite{FS2} 

\subsection{Affine Functions and the Hitzelberger-Lytchak Theorem}

Let $X$ be a geodesic metric space.
For $x,y \in X$ we denote by
$m(x,y)=\{z\in X \, | \, d(x,z)=d(z,y)=\frac{1}{2}d(x,y)\}$ the set of midpoints of $x$ and $y$.
A map $f:X \to Y$ between two geodesic metric spaces is called
{\em affine}, if
for all $x,y \in X$,
we have
$f(m(x,y)) \subset m(f(x),f(y))$.
Thus a map is affine if and only if it maps
geodesics parameterized proportionally to arclength into 
geodesics parameterized proportionally to arclength.
An affine map $f:X \to \R$ is called an affine function. 

\begin{definition}
Let $X$ be a geodesic metric space. We say that affine functions on $X$ 
{\it separate points},
if for every $x,y\in X$, $x\neq y$, there exists an affine function $f:X\longrightarrow \mathbb{R}$
with $f(x)\neq f(y)$.
\end{definition}

The following beautiful rigidity theorem, which is due to Hitzelberger and Lytchak, is a main tool in our argument.

\begin{theorem} (\cite{HL}) \label{theo-HL}
Let $X$ be a geodesic metric space. If the affine functions on $X$ separate points, then $X$
is isometric to a convex subset of a (strictly convex) normed vector space.
\end{theorem}

For a proof of a variant of this statement, cf. Section \ref{subsec-3point-segment}.

\subsection{Normed Vector spaces}

The other main ingredient when characterizing the $3$-point Ptolemy-circle and-segment spaces is
the following theorem due to Schoenberg.
\begin{theorem} (\cite{Sch}) \label{theo-Sch}
A normed vector space $(V,||\cdot ||)$ is a Ptolemy metric space if and only if it is Euclidean.
\end{theorem}

This, together with the fact that the Ptolemy condition is invariant under scaling, yields the 
\begin{corollary} \label{cor:sch}
An open subset of a normed vector space $(V,||\cdot ||)$ is a Ptolemy space, if and only if
$(V,||\cdot ||)$ is Euclidean.
\end{corollary}

%%%%%%%%%%%%%%%%%%%%%%%%%%%%%%%%%%%%%%%%%%%%%%%%%%%%%%%%%%%%%%%%%%%%%%%%%%%%%%%%%%%%%%%%%%%%%%%%%%%%%%%%%%%%%%%%%%%%%%%%%%%%%%%%%%%%%%%%%%%%%%%%%%%%%%%%%%%%%%%%%%

\section{Classification of Circles and Ptolemy segments}
\label{sec-charac-circles}

In this section we classify circles and segments in
Ptolemy spaces. We start with a classification up
to M\"obius equivalence.

\subsection{M\"obius classification of circles and segments}

We prove Theorem \ref{thm:moebch-circ} and its analogue
for Ptolemy segments.

\begin{theorem} \label{thm:moebch-segm}
Let $I$, $I'$ be Ptolemy segments.
Let $x_1,x_3$ be the boundary points of $I$ and
$x'_1,x'_3$ be the boundary points of $I'$.
Let in addition $x_2$ and $x'_2$ inner points
of $I$ and $I'$.
Then there exists a unique M\"obius homeomorphism
$\phi:I\to I'$ with $\phi(x_i)=x'_i$.
\end{theorem}

\begin{proof} (of Theorem \ref{thm:moebch-circ})
Define
$\phi_C:C\to \partial \Delta$ by
$\phi_C(t)=\crt(t,x_1,x_2,x_3)$.
Since $C$ is a Ptolemy circle the image of $\phi_C$
is actually in $\partial \Delta$.
The map is continuous and maps
$x_1$ to $\hat{e}_1=(0:1:1)$,
$x_2$ to $\hat{e}_2=(1:0:1)$ and
$x_3$ to $\hat{e}_3=(1:1:0)$.

One also easily checks that
$\phi^{-1}(\hat{e}_i)= x_i$.
Now $C\setminus \{x_1,x_2,x_3\}$ consists out
of three open segments
$I_1,I_2,I_3$, such that
$x_i$ and $x_j$ are boundary points of $I_k$,
and $x_k $ is not $I_k$ (here $\{i,j,k\}=\{1,2,3\}$).
Correspondingly 
$\partial\Delta \setminus \{\hat{e}_1, \hat{e}_2,\hat{e}_3\}$
consists of three open segments
$J_1,J_2,J_3$. 
Note that $J_i$ consists of all
$(a_1:a_2:a_3)\in \partial \Delta$, such that
$|a_i| > \max\{|a_j|,|a_k|\}$.
Since $\phi_C$ gives a bijection 
$x_i\leftrightarrow\hat{e}_i$, 
$\phi_C$ maps $I_i$ to $J_i$.
If $t \in I_i$, this implies that the equality in the triangle
inequality is written in the following way:
$$d(t,x_j) d(x_i,x_k) + d(t,x_k) d(x_i,x_j)  = d(t,x_i) d(x_j,x_k).$$

We now show that $\phi_C$ is injective.
Assume that $\phi_C(s)=\phi_C(t)$.
This implies that there exists
$\la > 0$ such that $d(s,x_i)=\la d(t,x_i)$ for 
$i=1,2,3$. In particular we have
\begin{equation} \label{eq:c1}
d(x_1,t)d(s,x_2)=d(x_1,s)d(t,x_2)
\end{equation}

We can also assume w.l.o.g. that
$d(t,x_3) \ge \max \{ d(t,x_1),d(t,x_2)\}$, and hence the same holds
for $s$. This implies that $s$ and $t$ are in the 
same component of
$C\setminus\{ x_1,x_2, x_3\}$ (namely the component $I_3$). 
Thus we have (eventually after permuting $t$ and $s$)
$$d(x_1,t)d(s,x_2)+d(t,s)d(x_1,x_2)=d(x_1,s)d(t,x_2),$$
which implies $d(t,s)=0$ because of (\ref{eq:c1}).

Since $C$ is homeomorphic to $S^1$ and 
$\phi_C:C\to \partial \Delta$ is injective and continuous,
it is also surjective and a homeomorphism.

Now the map
$\phi:C\to C'$, $\phi:=\phi_{C'}^{-1}\circ \phi_C$ is
a M\"obius homeomorphism and
maps $x_i$ to $x'_i$.
Assume on the other side that
$\psi:C\to C'$ is a M\"obius homeomorphism with $\psi(x_i)=(x'_i)$.
Then $\phi_C(t)=\crt(t,x_1,x_2,x_3)=\crt(\psi(t),x'_1,x'_2,x'_3)=\phi_{C'}(\psi(t))$,
which implies $\psi=\phi_{C'}^{-1}\circ \phi_C$.
\end{proof}

The proof of Theorem \ref{thm:moebch-segm} is completely analogous.
The map $\phi_C(t):=\crt(t,x_1,x_2,x_3)$ now maps
$I$ homeomorphically on the path in $\partial \Delta$, which
goes from $\hat{e}_1$ via $\hat{e}_2$ to $\hat{e}_3$.
Again $\phi=\phi_{C'}^{-1}\circ \phi_C$ is the required homeomorphism.

\subsection{Classification of Ptolemy segments up to isometry}

We now study Ptolemy segments and classify them up to
isometry.\\

We first consider the special case that one boundary point
of the segment is the point $\infty$. Let
$x\in X\setminus\{\infty\}$ be the other boundary point.
Let $x<s<t<\infty$ points on the segment (where the order is induced by
the homeomorphism of the segment to $[0,1]$).
Then the Ptolemy equality implies
$d(x,t)+d(t,s)=d(x,s)$. This says that
the segment is isometric to an interval and actually isometric
to $[0,\infty] \subset \R\cup\{\infty\}$. \\

We now consider a Ptolemy segment
$([0,1],d)$, with 
$R := d(0,1)$ a positive and finite number. Let
let $Q:= [0,\infty)\times [0,\infty) \subset \mathbb{R}^2$.
We define a map $\psi : [0,1]\to Q$ by $t\mapsto p_t=
\bigl( \begin{smallmatrix} a_t \\ b_t \end{smallmatrix} \bigr)$,
where $a_t= d(t,1)$ and $b_t=d(t,0)$.
Thus $p_t$ is a curve in $Q$ from
$e_R^1 =\bigl( \begin{smallmatrix} R \\ 0 \end{smallmatrix} \bigr)$
to
$e_R^2=\bigl( \begin{smallmatrix} 0 \\ R \end{smallmatrix} \bigr)$.

Note that the above Ptolemy condition applied to the quadruple
$0\le t_1\le t_2\le 1$ implies that
$d(t_1,t_2)= a_{t_1}b_{t_2}-b_{t_1}a_{t_2}=\langle Jp_1,p_2\rangle$,
where $\langle \, ,\,\rangle$ is the standard scalar product on $Q$ and $J$
the standard rotation $J \bigl( \begin{smallmatrix} a \\ b \end{smallmatrix} \bigr)
=\bigl( \begin{smallmatrix} -b \\ a \end{smallmatrix} \bigr)$.

Now for two points $p,q \in Q$ we have
$\langle Jp,q\rangle \ge 0$ iff $\arg(p) \le \arg(q)$.
As a consequence this implies that
$t\mapsto \arg(p_t)$ is a strictly increasing function.

This motivates that we consider for two points
$p,q \in Q$ the expression
$\langle Jp,q\rangle$ as a kind of "signed distance".
This "signed distance" is related to the Ptolemy equality, as
a trivial computation shows:

\begin{lemma} \label{lem:PT-condition}
For $p_1, p_2, p_3, p_4 \in  Q$ (actually for
arbitrary $p_1, p_2, p_3, p_4 \in \R^2$) we have
$$\langle Jp_1,p_2\rangle \langle Jp_3,p_4\rangle +
\langle Jp_2,p_3\rangle \langle Jp_1,p_4\rangle
= \langle Jp_1,p_3\rangle \langle Jp_2,p_4\rangle$$
\end{lemma}

\begin{remark}
Thus the expression
$R(p_1,p_2,p_3,p_4)= \langle Jp_1,p_2\rangle \langle Jp_3,p_4\rangle$
has the symmetries of a curvature tensor. The above Lemma corresponds to
the Bianchi identity. The other symmetries are obvious.
\end{remark}

The expression $\langle Jp,q\rangle$ does not satisfy the triangle inequality
and we have to study more precisely what conditions correspond to
the triangle inequality.
Let us therefore consider three points $u,v,w \in Q$ such that
$\arg(u) \le \arg(v) \le \arg(w)$. This implies
that $v= \la \, u + \mu \, w$, with $\la, \mu \ge 0$.
We say that $u,v,w$ {\em satisfies the triangle inequality},
if the following three inequalities hold:

\begin{enumerate}
\item $\langle Ju,v\rangle + \langle Jv,w\rangle \ge \langle Ju,w\rangle$
\item $\langle Jv,w\rangle + \langle Ju,w\rangle \ge \langle Ju,v\rangle$
\item $\langle Ju,w\rangle + \langle Ju,v\rangle \ge \langle Jv,w\rangle$
\end{enumerate}

Clearly for $0 \le t_1\le t_2 \le t_3\le 1$ the points
$p_{t_1},p_{t_2},p_{t_3}\in Q$ have to satisfy the triangle inequality.

An easy computation shows the following:

\begin{lemma} \label{lem:trianglecondition}
Let $u,v,w \in Q$ as above. Then the three triangle inequalities are respectively equivalent to:
$\la + \mu \ge 1$, $\la +1 \ge \mu$ and $1+\mu \ge \la$, i.e. that the three non-negative numbers
$\la ,\mu ,1$ satisfy the triangle inequality. 
\end{lemma}

Let $u,w\in Q$ with $\arg(u) < \arg(w)$. We define the region
$$T(u,w) =\{ (\la u +\mu w)\in Q  |\, 0\le \la,\, 
0\le \mu,\, \la +\mu \ge 1,\, \la +1\ge \mu,\, \mu +1 \ge \la \}.$$

Note that $T(u,w)$ is the convex region in $Q$ which is bounded
by the line segment $s\, u + (1-s)\, v$ , where $0\le s\le 1$ from $u$ to $w$, and the two parallel rays
$u + s\,(u+v)$ resp. $w + s\,(u+v)$ for $0\le s < \infty$.
These three affine segments correspond to the three triangle inequalities.
This can easily be verified,
since the equation $\la + \mu =1$ 
defines the line $\ell(u,w)$, 
the equation $ \la + 1= \mu$ the line
$\ell_w=\ell(w,w+(u+w))$ and $\mu +1 =\la$ the line
$\ell_u=\ell(u,u+(u+w))$. The last two lines are parallel,
The inequalities define corresponding halfspaces.
Here we denote for different
$p,q \in Q$ with $\ell(p,q)$ the affine line
determined by $p$ and $q$.

If we consider a Ptolemy interval, then the corresponding curve
$p_t$ satisfies: if $0\le t_1<t_2<t_3\le 1$ then
$p_{t_2}\in T(p_{t_1},p_{t_3})$.
In particular the whole curve is contained in the set
$T(e_R^1,e_R^2)$.

The first of the three inequalities (namely $\la + \mu \ge 1$) 
is easy to understand.
It just means that the curve $p_t$ is convex in the following sense.

\begin{definition}
We call a curve $p_t$ in $Q$ from 
$e_R^1$
to $e_R^2$ {\em convex}, if
$p_t$ is continuous,
$\arg(p_t)$ is strictly increasing and the bounded component
of $Q\setminus \{p_t | t\in [0,1]\}$ is convex.
\end{definition}

Surprisingly this condition together with the condition
that $p_t \in T(e_R^1,e_R^2)$ imply all other triangle conditions.

\begin{lemma} \label{prop:T-condition}
Let $p_t$ be a convex curve in $T(e_R^1,e_R^2)$ from
$e_R^1$ to $e_R^2$, then for $0\le t_1<t_2<t_3\le 1$ we have
$p_{t_2}\in T(p_{t_1},p_{t_3})$.
\end{lemma}

\begin{proof}
Let $p_t$ be a convex curve in $T(e_1,e_2)$ from
$e_1$ to $e_2$.
We denote
$u=p_{t_1}, v=p_{t_2}, w=p_{t_3}$.
We have to show that
$v \in T(u,w)$.
Consider the lines 
$\ell(e_1,u)$ and $\ell(e_2,w)$.

There are two cases.
First assume that the lines are parallel
(and hence do not intersect).
In this case $u$ and $w$ lie on the two boundary rays of
$T(e_R^1,e_R^2)$ and one easily checks that
$T(u,w)$ is the unbounded component of $T(e_R^1,e_R^2) \setminus \ell(u,w)$.
Since $p_t$ is convex, this implies that $v \in T(u,w)$.

Thus we consider the second case that the two lines intersect in a point
$z$. Since $p_t$ is a convex curve containing $e_R^1, u,w,e_R^2$, we see that
$z$ is in the unbounded component of $T(e_R^1,e_R^2) \setminus \{p_t| t\in [0,1] \}$ and
$v \in \Delta(u,w,z)$, where $\Delta(u,w,z)$ is the corresponding triangle.

It remains to prove that $z \in T(u,w)$.
Let $\al = \arg(u+w)=\angle_o(e_R^1,u+w)$, where $o$ is the origin, and
$\angle$ the usual Euclidean angle.
Let $\beta=\angle_o(u+w,e_R^2)$. Thus $\al +\beta =\pi/2$.
We assume (without loss of generality) that $\al \ge \pi/4$.
Consider the two lines
$\ell_u=\ell(u,u+(u+w))$ and $\ell_w= \ell(w,w+(u+w))$, which are the lines containing 
the boundary rays of $T(u,w)$.
Then $\ell_u$ intersect the boundary of $Q$ in a point
$q_1=\bigl( \begin{smallmatrix} a \\ 0 \end{smallmatrix} \bigr)$
and $\ell_w$ intersects the boundary of $Q$ in
$q_2=\bigl( \begin{smallmatrix} 0 \\ b \end{smallmatrix} \bigr)$.
Since $\al \ge \pi/4$, we have $b \le a$.
Since $\angle_{q_2}(w,2q_2) = \beta \le \pi/4$ and
$\angle_{e_R^2}(w,2e_R^2) \ge \pi/4$ (since $w \in T(e_R^1,e_R^2)$),
we see that $b\le 1$ and hence also $a \le 1$.
This implies that $z =\ell(e_R^1,u)\cap \ell(e_R^2,w)$ is in the strip
bounded by $\ell_u$ and $\ell_w$ and hence in $T(u,w)$.
\end{proof}

Collecting all results we can now state:
\begin{proposition} \label{prop-segment}
The isometry classes of Ptolemy intervals
$([0,1],d)$ with $d(0,1)=R$ stay in
$1-1$ relation to the convex curves in $T(e_R^1,e_R^2)$ from $e_R^1$ to $e_R^2$
modulo reflection at the bisecting line in $Q$.
\end{proposition}

Indeed given such a Ptolemy interval, we obtain such a convex curve.
If we have otherwise given such a convex curve
$\psi:[0,1] \to Q$, $\psi(t)= p_t=
\bigl( \begin{smallmatrix} a_t \\ b_t \end{smallmatrix} \bigr)$,
then for given $s \le  t \in [0,1]$ we have
$d(0,s)=b_s$, $d(1,s)=a_s$, $d(0,t)=b_t$, $d(1,t)=a_t$.
Now $d(t,s)$ is determined by the Ptolemy equality:
$$d(s,t)d(0,1)+d(0,s)d(t,1)=d(0,t)d(s,1).$$
Thus the curve determines the isometry class completely. \\
Note that two Ptolemy segments are isometric to each other if and only if
their Ptolemy parameterizations as above either coincide or if they are
obtained from one another by reflection at the bisecting line in $Q$. \\

\begin{remark} \label{rem-segments}
\begin{description}
\item [(i)] On Ptolemy intervals which can be realized in $\mathbb{E}^2$. \\
Let $p,q\in\mathbb{E}^2$ be two points of distance
$|pq|=R$.
Then for $r= \frac{R}{2}$
there is exactly one Ptolemy segment (namely a halfcircle of radius r) in $\mathbb{E}^2$ connecting two points of distance $R$ up to isometry, whereas
for $r>\frac{R}{2}$ there are exactly two such segments. These segments are precisely the 
Ptolemy segments which can be isometrically 
embedded in $\mathbb{E}^2$
(this follows essentially from the classical characterization of
circles by the Ptolemy equality). We claim now, that their images in $Q$ are the intersections of $Q$ with the ellipses
in $\mathbb{R}^2$ through $e_R^1$ and $e_R^2$ with axis along  
$\operatorname{span}\{ e_R^1+e_R^2 \}$ and $\operatorname{span}\{ e_R^1-e_R^2\}$. Indeed, for $R>0$, $r\ge \frac{R}{2}$ 
consider a circle in $\mathbb{E}^2$ with radius $r$ connecting two points $p$ and $q$ in distance $|pq|=R$ of each other.
Seen from the circle`s origin, the points $p$ and $q$ enclose an angle $\alpha$ with $\sin \frac{\alpha}{2}=\frac{R}{2r}$.
In each point $m$ on the circle segment considered, let $\beta$ denote the angle of $p$ and $q$ at $m$.
This angle $\beta =[\pi -\frac{\alpha}{2}]$ is constant along the segment and using the law 
of cosine in $\mathbb{E}^2$ the image of this circle segment
in $Q$ is given in the coordinates $a$ and $b$ of $Q$ through $R^2=a^2+b^2-2ab \cos \beta$.
Since the last equation describes the ball of radius $R$ around the origin w.r.t. the scalar product on $\mathbb{R}^2$ with $<e_R^1,e_R^1>=1=<e_R^2,e_R^2>$
and $<e_R^1,e_R^2>=-\cos \beta$, it determines an ellipse as claimed. \\
This remark shows that the possible isometry classes
of Ptolemy segments are much richer than those of circles which can be
realized in the Euclidean plane. 
\item[(ii)] The Ptolemy parameterization. \\
Given a Ptolemy segment $(I,d)$, there exists a unique {\rm Ptolemy parameterization}
$\mathfrak{pt}:[0,\frac{\pi}{2}]\longrightarrow Q$ as above, parameterizing the segment
by the angle $\alpha \in [0,\frac{\pi}{2}]$ its image point in $Q$ encloses with $e_1^R$. \\
Consider two Ptolemy segments $I_i$, $i=1,2$, and denote their Ptolemy parameterizations by
$\mathfrak{pt}_i$, $i=1,2$. Then the map $\varphi: I_i\longrightarrow I_2$ satisfying
$\alpha (\varphi (x))=\alpha (x)$ for all $x\in I_1$ is a M\"obius map. 
\end{description}
\end{remark}

%%%%%%%%%%%%%%%%%%%%%%%%%%%%%%%%%%%%%%%%%%%%%%%%%%%%%%%%%%%%
\begin{figure}[htbp]
\centering
%\psfrag{E2}{$\mathbb{E}^2$}
%\psfrag{Q}{$Q$}
%\psfrag{R}{$R$} 
\includegraphics[width=0.9\columnwidth]{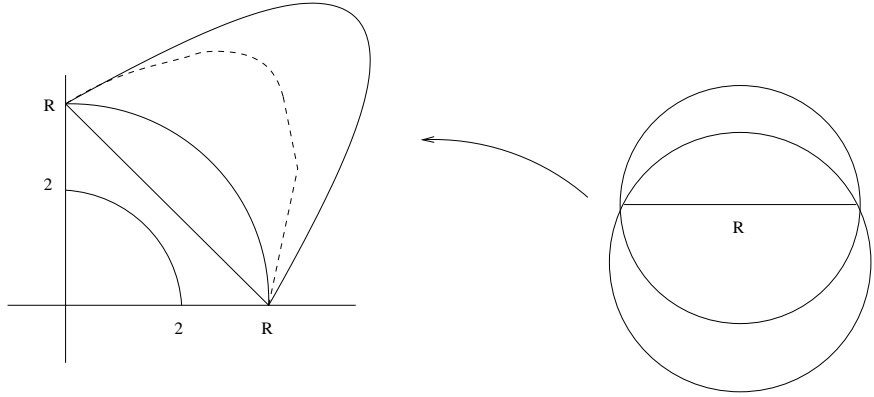}
\caption{The figure shows a variety of Ptolemy parameterizations of different Ptolemy segments. The dashed curv
on the left hand side corredponds to a non Euclidean configuration. The others are pieces of ellipses as described
in Remark \ref{rem-segments}. Their corresponding Euclidean configurations are shown on the right hand side.}
\end{figure}
%%%%%%%%%%%%%%%%%%%%%%%%%%%%%%%%%%%%%%%%%%%%%%%%%%%%%%%%%%%%

\subsection{Classification of Ptolemy circles up to isometry}

We now study Ptolemy circles.
Let $S^1=\{ e^{\pi it}\in \C\,|\,0\le t \le 2\}$.
We will assume that $d(-1,1)=R$. Now $S^1$ consists of
two "segments" from $1$ to $-1$.
Let $H= [0,\infty)\times \R$ be the upper halfspace.
We now define a map 
$\varphi : [0,2] \to H$, by 
$t \mapsto p_t=
\bigl( \begin{smallmatrix} a_t \\ b_t \end{smallmatrix} \bigr)$,
where $b_t=d(e^{\pi it},1)$ and
$a_t = d(e^{\pi it},-1)$ for $0\le t \le 1$ and
$a_t = - d(e^{\pi it},-1)$ for $1\le t \le 2$. 

Let $(S^1,d)$ be a Ptolemy circle and
let $0 \le t_1<t_2\le 2$. Then the Ptolemy condition for the circle
applied to the three possible cases
$0\le t_1<t_2 \le 1$, $0\le t_1 \le 1\le t_2$ and
$0 < 1\le t_1\le t_2$ always gives the distance
$d(t_1,t_2) = \langle J p_{t_1},p_{t_2}\rangle$.

This implies that $\arg (p_t)$ is  strictly increasing with $t$.
The discussion with the triangle inequality is similar as in the case
of a Ptolemy segment. 
However, note that in the proof of Lemma \ref{lem:trianglecondition} we used
the fact that $\langle Ju,v\rangle \neq 0$. Thus the argument does not work
for $u=e_R^1$ and $w=-e_R^2$. We can, however, say the following:
if $0 < t_1<t_2<t_3 < 2$ are three points in the open interval $(0,2)$, then
the triangle condition is equivalent to
$p_{t_2} \in T(p_{t_1},p_{t_3})$ as above. 
The same is true if $t_1=0$ and $t_3<2$ (resp. $0<t_1$ and $t_3=2$).

We want to understand the limit case $t_1= 0, t_3=2$.

Therefore we define for a unit vector $x\in S^1$ with $0\le \arg(x) \le \pi$
the sector
$T_x(e_R^1,-e_R^1) := \{ (s e_R^1 + t x) \,|\; -1\le s\le1, 0\le t<\infty \}$.

The analogon of Proposition \ref{prop-segment} for Ptolemy circles now reads as follows.

\begin{proposition} \label{prop-circle} 
The Ptolemy circles $(S^1,d)$ with $d(1,-1)=R$ are in $1-1$
relation to the convex curves $p_t$ in $H$ from $e_R^1$ via $e_R^2$ to
$-e_R^1$ which are contained in $T_x(e_R^1,-e_R^1)$ for
some $x \in S^1$ with $\pi/2 \le x \le 3\pi/2$.
\end{proposition}

\begin{proof}
The condition that (for $0\le t_1<t_2<t_3< 2$)
$p_{t_2} \in T(p_{t_1},p_{t_3})$ implies
that the curve $p_t$ convex, i.e.
the bounded component of $H\setminus \{p_t | \,0\le t\le 2\}$
is convex. The condition also implies (for $t_1=0$ and $t_3\to 2$ a limit condition
that $p_{t_2} \in T_x(e_1,-e_1)$ for some $x \in S^1$. Since
$e_2 = p_{1}$ we have $\pi/2 \le x \le 3\pi/2$.

Conversely let us assume that 
$p_t=\bigl( \begin{smallmatrix} a_t \\ b_t \end{smallmatrix} \bigr)$, 
$0\le t\le 2$ is a convex curve in $H$ from $p_0=e_1$ via 
$p_1=e_R^2$ to
$p_2=-e_R^1$ which is contained in $T_x(e_R^1,-e_R^1)$ for
some $x \in S^1$ with $\pi/2 \le x \le 3\pi/2$.
For $0\le t_1\le t_2\le 2$ define 
$d(e^{\pi it_1},e^{\pi i t_2})=\langle Jp_{t_1},p_{t_2}\rangle$.
By Lemma \ref{lem:PT-condition} $(S^1,d)$ satisfies the Ptolemy condition.
We have to show that $d$ is indeed a metric, hence that for 
$0\le t_1<t_2<t_3< 2$ we have
$p_{t_2} \in T(p_{t_1},p_{t_3})$.

The proof is similar to the proof of
Proposition \ref{prop:T-condition}

We write
$u=p_{t_1}, v=p_{t_2}, w=p_{t_3}$
and have to show that
$v \in T(u,w)$.
Consider the lines  
$\ell(e_R^1,u)$ and $\ell(-e_R^1,w)$.

If these lines are parallel, then they are the boundary rays of
$T_x(e_R^1,-e_R1)$ and $T(u,w)$ is the closure of the unbounded component of
$T_x(e_R^1,-e_R^1) \setminus \ell(u,w)$.

If the lines are not parallel, let $z =\ell(e_R^1,u)\cap \ell(-e_R^1,v)$ be the intersection point
and we have to show $z\in T(u,w)$.

Consider the parallel lines
$\ell_u=\ell(u,u+(u+w))$ and $\ell_w= \ell(w,w+(u+w))$, which are the lines containing 
the boundary rays of $T(u,w)$. These lines intersect the boundary of $H$ in points
$-\la e_1$ and $\la e_1$ for some $0\le \la$.
The condition that $u,w \in T_x(e_R^1,-e_R^1)$ implies
$\la \le 1$. This implies that $z$ is in the strip bounded by
$\ell_u$ and $\ell_w$.
\end{proof}

%%%%%%%%%%%%%%%%%%%%%%%%%%%%%%%%%%%%%%%%%%%%%%%%%%%%%%%%%%%%
%\begin{figure}[htbp]
%\centering
%\psfrag{2}{$2$}
%\psfrag{R}{$R$} 
%\includegraphics[width=0.9\columnwidth]{parameter2}
%\caption{The figure shows a variety of different Ptolemy circles.}
%\end{figure}
%%%%%%%%%%%%%%%%%%%%%%%%%%%%%%%%%%%%%%%%%%%%%%%%%%%%%%%%%%%%

\begin{remark} \label{rem-circles}
\begin{description}
\item[(i)] On the Ptolemy circles which can be realized in $\mathbb{E}^2$. \\
Now we can have a similar discussion of the variety of Ptolemy circles as the one in Remark \ref{rem-segments} (i).
for Ptolemy segments. 
Once again, the Ptolemy circles which admit isometric embeddings in $\mathbb{E}^2$ are precisely given by the intersections
of the upper halfplane $H$ with the ellipses in $\mathbb{R}^2$ considered above. 
\item[(ii)] The Ptolemy parameterization. \\
The analogon of Remark \ref{rem-segments} (ii) also applies to Ptolemy circles once one 
fixes additional two points on the circle, fixing its Ptolemy parameterization. \\
Note that in contrast to the situation for Ptolemy segments, Ptolemy circles do not
admit the pair of an initial- and an endpoint. Our characterization of isometry classes
therefore requires the additional choice of two distinct points on the circle. It therefore is a 
characterization of two-pointed Ptolemy circles.
\end{description}
\end{remark}

%%%%%%%%%%%%%%%%%%%%%%%%%%%%%%%%%%%%%%%%%%%%%%%%%%%%%%%%%%%%%%%%%%%%%%%%%%%%%%%%%%%%%%%%%%%%%%%%%%%%%%%%%%%%%%%%%%%%%%%%%%%%%%%%%%%%%%%%%%%%%%%%%%%%%%%%%%%%%%%%%%
%%%%%%%%%%%%%%%%%%%%%%%%%%%%%%%%%%%%%%%%%%

%%%%%%%%%%%%%%%%%%%%%%%%%%%%%%%%%%%%%%%%%%%%%%%%%%%%%%%%%%%%%%%%%%%%%
%%%%%%%%%%%%%%%%%%%%%%%%%%%%%%%%%%%%%%%%%%%%%%%%%%%

\section{Ptolemy spaces with many circles}
\label{sec-many circles}

In this section we characterize Ptolemy spaces with many circles or
many segments. For simplicity we call a metric space a
$k$-point Ptolemy-circle (resp. Ptolemy-segment) space,
if any $k$ points are contained in a Ptolemy ccircle (respectively Ptolemy segment).

\subsection{$4$-Point Ptolemy-Circle and Segment Spaces}

\label{subsec-4point}

In this section we characterize all $4$-point Ptolemy circle/segment spaces up to M\"obius equivalence. 

We assume first, that $X$ is a space such that for any allowed
quadruple $(x_1,x_2,x_3,x_4)$ we have
$\crt(x_1,x_2,x_3,x_4) \in \partial \Delta$, i.e. for all quadruples
of points the Ptolemy equality holds.
We choose some point $p \in X$ and consider the
metric space $(X\sm\{ p\},d_p)$.
The Ptolemy equality implies now that for a triple of points in $X\sm\{ p\}$ 
we have equality in the triangle inequality.

Now we use the following elementary characterization of subsets of the real line,
which we leave as an exercise.

\begin{lemma} \label{lem:triangle}
A metric is isometric to a subset of the real line $\R$ if and only if
for any triple of points we have equality in the triangle inequality.
\end{lemma}

Thus $X$ is M\"obius equivalent to $Y\cup\{\infty\}$, where
$Y \subset \R$. Since we always assume that $X$ contains at least two points
and Ptolemy segment spaces are connected, we immediately obtain
Corollary \ref{cor-4points}.

%%%%%%%%%%%%%%%%%%%%%%%%%%%%%%%%%%%%%%%%%%%%%%%%%%%%%%%%%%%%%%%%%%%%%%%%%%%%%%%%%%%%%%%%%%%%%%%%%%%%%%%%%%%%%%%%%%%%%%%%%%%%%%%%%%%%%%%%%%%%%%%%%%%%%%%%%%%%%%%%%%

\subsection{$3$-Point Ptolemy-Circle Spaces}
\label{subsec-3point-circle}

In this section we characterize all $3$-point Ptolemy circle spaces up to M\"obius equivalence and prove Theorem \ref{main-theorem}. \\

The idea of the proof is the following. First
we take an arbitrary point $z \in X$ and consider the
metric space $(X\sm \{ z\}, d_z)$. This is a geodesic space
by the results of Section \ref{subsec:uofptcirc}. We show that on
this space the affine functions separate points. Thus, by the Hitzelberger-Lytchak Theorem,
$(X\sm\{z\},d_z)$ is isometric to a convex subset of a normed vector space.
Actually since all geodesics extend to lines it is isometric to a normed vector space.
By Schoenberg's Theorem this vector space is an inner product space and hence
by local compactness isometric to $\mathbb{E}^n$, which completes the proof.

\begin{lemma} \label{lemma-affine}
Let $X$ be a compact 3-point Ptolemy circle space and let $z\in X$. Then affine functions
separate points on the geodesic space $(X\setminus \{ z\} , d_z)$. 
\end{lemma}
\begin{proof}

We will show below that the Busemann function $b_c$ associated to a geodesic ray $c$ in  $(X\setminus \{ z\} ,d_z)$ is affine.
Assume this for a moment, and let 
$x,y \in  X\setminus \{ z\}$, $x\neq y$ be distinct points. 
Note that by the results of Section \ref{subsec:uofptcirc} the geodesic
between $x$ and $y$ is a part of the unique geodesic circle
through $x,y,z$ whose part
in $(X\sm \{ z\}, d_z)$ is a geodesic line $h$. This line contains $x$ and $y$ and
the Busemann function of this line separates these points,
$b_h(x)\neq b_h(y)$.  
\end{proof}

It remains to show

\begin{proposition} \label{pro:Busemann_convex_circle}
Busemann functions $b:X\sm\{z\}\to \R$ are affine.
\end{proposition}

\begin{proof}

Let $c:[0,\infty )\longrightarrow X\setminus \{ z\}$ be an arbitrary geodesic ray in  $(X\setminus \{ z\} , d_z)$.
In order to prove that the Busemann function $b:=b_c$ is affine, 
consider $x,y,m \in X$ with
$a:=d(x,y)/2=d(x,m)=d(m,y)$.
Since Busemann functions are convex (cf. Section \ref{sec:bus_convex}),
we have $b(m)\le \frac{1}{2}(b(x)+b(y))$.

To prove that $b$ is affine, we need to prove the opposite inequality,
i.e. we have to prove:

\begin{equation} \label{eq:0}
b(m)\ge \frac{1}{2}(b(x)+b(y))
\end{equation}

Consider $w_i=c(i)$ for $i\to \infty$. By assumption there exists a
Ptolemy circle $\sigma_i$ containing $x,y,w_i$.

We consider the subsegment of this circle
which contains $x$ and $w_i$ as boundary points and $y$ as interior
point. On this segment we consider the orientation that
$x<y<w_i$ and we choose for $i$ large enough $u_i\in \sigma_i$ with
$x<y<u_i<w_i$, such that $d(y,u_i)=a$, and hence
$d(x,u_i)\le 3a$. By the Ptolemy equality on the segment we have
$$3a\,	d(y,w_i)\ge d(y,w_i)\, d(x,u_i) = 2a\, d(u_i,w_i)+\, a\, d(x,w_i)$$
and hence
\begin{equation} \label{eq:1}
d(y,w_i)\ge \frac{2}{3}d(u_i,w_i)+ \frac{1}{3} d(x,w_i) .
\end{equation}

Since $(X\sm\{z\})$ is locally compact,
some subsequence converges, thus we can assume $u_i\to u$ and 
$x,y,u$ lie on a geodesic.

Since $u_i\to u$, the Inequality (\ref{eq:1})
implies in the limit for the Busemann function

\begin{equation} \label{eq:2}
b(y)\ge \frac{2}{3}b(u)+ \frac{1}{3} b(x).
\end{equation}

Since $x,y,u$ are on a geodesic and $d(y,u)=a$,
we have $d(x,u)=3a$. By triangle inequalities this implies
$d(m,u)=2a$, and hence $y$ is a midpoint of $m$ and $u$.
This again implies
\begin{equation} \label{eq:3}
b(y)\le \frac{1}{2}(b(m)+b(u)).
\end{equation}

Now an easy computation shows that Inequalities
(\ref{eq:2}) and (\ref{eq:3})
imply the desired
estimate (\ref{eq:0}).
\end{proof}

Lemma \ref{lemma-affine} now allows us to provide the \\

{\bf Proof of Theorem \ref{main-theorem}:} From Theorem by \ref{theo-HL}, we
deduce with Lemma \ref{lemma-affine} that $(X\setminus \{ z\},d_z)$ 
is a convex subset of 
a (strictly convex) normed vector space. 
Furthermore, all geodesic segments extend to complete lines.
Thus we see that $(X\sm\{z\},d_z)$, is itself isometric to a normed vector space.
Since it also is a Ptolemy space, Schoenberg's Theorem, Theorem \ref{theo-Sch}, implies 
that it is a Euclidean space. 
By local compactness it is isometric to $\mathbb{E}^n$ for some $n\in \N$.
\begin{flushright}
$\Box$
\end{flushright}

%%%%%%%%%%%%%%%%%%%%%%%%%%%%%%%%%%%%%%%%%%%%%%%%%%%%%%%%%%%%%%%%%%%%

\subsection{An Example of a M\"obius Sphere}
\label{subsec-exotic}

In this section we provide an example of a metric sphere which is M\"obius equivalent, but
not homothetic to the standard chordal sphere. \\

Let $X$ be a $\operatorname{CAT}(\kappa )$-space, $\kappa <0$, and let $o\in X$. We recall that the Bourdon metric 
$\rho_o:\partial_{\infty}X\times \partial_{\infty}X\longrightarrow \mathbb{R}_0^+$
can also be expressed in terms of the Gromov product on $\partial_{\infty}X$. \\
Given $x,y\in X$, the {\em Gromov product} $(x\cdot y)_o$ of $x$ and $y$ w.r.t. the basepoint $o$
is defined as
\begin{displaymath}
(x\cdot y)_o \; := \; \frac{1}{2} \Big[ d(x,o) \, + \, d(y,o) \, - \, d(x,y) \Big] .
\end{displaymath}
This Gromov product naturally extends to points at infinity, by
\begin{displaymath}
(\xi \cdot \xi')_o \; := \; \lim\limits_{i\longrightarrow \infty} \, (x_i \cdot x_i')_o \hspace{0.5cm}
\forall \xi , \xi' \in \partial_{\infty}X,
\end{displaymath}
where $\{ x_i \}$ and $\{ x_i'\}$ are sequences in $X$ converging to $\xi$ and $\xi'$, respectively,
i.e. sequences which satisfy $\lim\limits_{i\rightarrow \infty} (\gamma_{o,\xi}(i) \cdot x_i)_o =\infty$
and  $\lim\limits_{i\rightarrow \infty} (\gamma_{o,\xi'}(i) \cdot x_i')_o =\infty$, respectively,
where $\gamma_{o,\eta}$ denotes the unique geodesic ray connecting $o$ to $\eta \in \{ \xi ,\xi'\}$. \\
This limit exists and does not depend on the choice of sequence. This phenomenon is referred to as the
so called {\rm boundary continuity} of $\operatorname{CAT}(\kappa )$-spaces (cf. Section 3.4.2 in \cite{buys}
and note that one can generalize the proof given there for proper $\operatorname{CAT}(\kappa)$-spaces
to the non-proper case). \\
With this notation one can write the Bourdon metric $\rho_o$ as 
\begin{displaymath}
\rho_o(\xi ,\xi') \; = \; e^{-\sqrt{-\kappa}(\xi \cdot \xi)_o} \hspace{1cm} \forall \xi ,\xi'\in \partial_{\infty}X 
\hspace{1cm} \mbox{(cf. \cite{FS1})}.
\end{displaymath}

\begin{example}
\label{example-exotic-m-s}
Consider the $3$-dimensional real hyperbolic space $\mathbb{H}^3_{-1}$
of constant curvature $-1$ in the Poincar\'e ball model. Let $o$ denote
the center of the ball and consider a complete geodesic $\gamma$ through $o$. \\
Now we glue a real hyperbolic halfplane $H$ of curvature $-1$ along $\gamma$.
The resulting space $X$ is a $\operatorname{CAT}(-1)$-space. \\
Let $o'\in H$ be some point, the projection of which in $H$ on $\gamma$ 
coincides with $o$. 
The Bourdon metric $\rho_o$ of $\partial X$ w.r.t. $o$ when restricted to the 
boundary of $\mathbb{H}^3_{-1}$, $S^2=\partial_{\infty}\mathbb{H}^3_{-1}\subset \partial_{\infty}X$,
is isometric to $S^2$ when endowed with half of its chordal metric. \\
The Bourdon metric $\rho_{o'}$ of $\partial X$ w.r.t. $o'$ when restricted to  
$\partial_{\infty}\mathbb{H}^3_{-1}$ is M\"obius equivalent to $\rho_o$. We now verify
that $(S^2,\rho_{o'})$ is not homothetic to the standard chordal sphere $(S^2,d_0)$. \\

Let $N=\{ \gamma (i)\}_{i}$ and $S=\{ \gamma (-i)\}_{i}$ denote the endpoints of $\gamma$ in $S^2$. They are diametral points w.r.t.
$\rho_o$ and define the equator $A$ as such sets of points with coinciding distances to
$N$ and $S$, respectively. Note that by the symmetry of the construction, $A$ has the very
same property w.r.t. the metric $\rho_{o'}$ \\
Since every geodesic ray in $X$ from $o'$ to some $a\in A$ contains $o$, $A$ endowed with the
Bourdon metric $\rho_{o'}$ with respect to $o'$ is isometrically a Euclidean circle of radius
$e^{-l}$, where $l:=|oo'|$. \\
In contrast to the points $a\in A$, the points $N$ and $S$ also lie in the boundary of the
halfplane $H$. Hence the geodesic rays connecting $o'$ to $N$ and $S$, respectively, do not
contain $o$. In fact, denote by $b_{\gamma}$ the Busemann function associated to $\gamma$ normalized
such that $b_{\gamma}(o)=0$, then 
\begin{displaymath}
\rho_{o'} (N,S) \; = \; e^{-(N\cdot S)_{o'}} \; = \; e^{-b_{\gamma}(o')}> e^{-l}. 
\end{displaymath}
It follows that $(S^2,\rho_{o'})$ is not homothetic to $(S^2,d_0)$.
\end{example}

%%%%%%%%%%%%%%%%%%%%%%%%%%%%%%%%%%%%%%%%%%%%%%%%%%%%%%%%%%%%%%%%%%%%%%%%%%%

\subsection{$3$-Point Ptolemy-Segment Spaces}
\label{subsec-3point-segment}

In this section we consider $3$-point Ptolemy segment spaces and we 
will prove Theorem \ref{theo-3pt-int}.

The proof of this result is surprisingly much more involved 
than the proof of Theorem \ref{main-theorem}. The idea of the proof is the same,
but there arise quite a number of technical problems.

The main step is to reduce the problem to the
case, that $X$ is already a subset
of the classical space $(\mathbb{S}^n,d_0)$. 
For the classical case our result can also be reformulated
in the following way:

\begin{proposition} \label{prop:eucl}
Let $X\sub \R^k\cup\{\infty\}$ with $\infty \in X$ such
that $X\sm\{\infty\}$ is a closed subset of $\R^k$.
Assume that through any three distinct points of $X$ there
exists a circle segment through these points.
Then $X\sm \{\infty\}$ is contained in an affine subspace
$H\sub \R^k$ (which could be $\R^k$ again) and 
$H\cap X$
is isometric 
either to $H$, to some closed halfspace in $H$ or
to the complement of some open distance ball in $H$.
\end{proposition}

We will give a scetch of the proof at the end of the section.

We now come to the proof of Theorem \ref{theo-3pt-int}. \\

Let $(X,d)$ be a compact $3$-point Ptolemy segment space, and consider
again a chosen point $z\in X$ and the locally compact
metric space $M:=(X\sm\{z\},d_z)$. 

Let $x,y\in M$ be distinct points.
By assumption there exists a Ptolemy segment $\sigma$ in $X$ containing
$x,y,z$. Using the arguments of Section \ref{subsec-4point} and
Lemma \ref{lem:triangle} we see that
$\sigma\sm\{z\}$ is in the metric space $M$ isometric to a subset of
the real line $\R$. Thus using some parameterization by arclength
there exists an isometric map $c_{xy}:\R\sm I\to M$, with
$c_{xy}(0)=x$, $c_{xy}(d(x,y))=y$, where $I\sub \R$ is either empty or an
open interval. 
We call the image $c_{xy}(\R\sm I) \sub M$ an
{\em geodesic minus interval} through $x$ and $y$. 
We do not know uniqueness in the moment. The geodesic minus
interval contains at least one ray.
It may be that $I \sub [0,d(x,y)]$ and in this case,
there may not be a geodesic between $x$ and $y$. Thus
$M$ is in general not a geodesic space.
In particular we can not directly use the results of section
\ref{subsec:uofptcirc}. Even if a geodesic exists between two points, we do not 
know uniqueness. We can also not use the Hitzelberger Lytchak theorem.
Instead we have to reprove the relevant results in our special situation.

We first have to state some properties of affine functions.
Let $M$ be a (not necessarily geodesic) metric space. In our context $M$ is
the space $(X\sm\{z\},d_z)$.
For $x,y \in M$ we denote by
$m(x,y)=\{z\in M \, | \, d(x,z)=d(z,y)=\frac{1}{2}d(x,y)\}$ the set of midpoints of $x$ and $y$.
A function $f:X \to \R$ is called
{\em affine}, if
for all $x,y \in X$ and all $m \in m(x,y)$,
we have
$f(m) =\frac{1}{2}(f(x)+f(y))$. \\

Essential for our argument is that Busemann functions of geodesic
rays are affine in this sense.

\begin{proposition} \label{pro:Busemann_convex_segment}
A Busemann function $b:M\to \R$ is affine in this sense.
\end{proposition}

We postpone the proof and focus on the consequences. \\

We set
$$\cA'(M):=\{f:M\to \R|\ f\ \mbox{affine and Lipschitz}\}.$$
$\cA'(M)$ is a Banachspace, where $\| f\|$ is the optimal Lipschitz constant.
We set $\cA(M):=\cA'(M)/\sim$, where $f\sim g$ if $(f-g)$ is constant.
By $[f]$ we denote the equivalence class of $f$.
Then also $\cA(M)$ is a Banachspace, where $\|[f]\|=\|f\|$. \\
Let $\cA^*(M)$ be the Banach dual space of $\cA(M)$ with the norm
$$\|\rho\|=sup_{[f]\in\cA(M)}\frac{|\rho([f])|}{\|[f]\|}.$$
For $x,y \in M$ let $E(x,y)\in \cA^*(M)$ 
be the evaluation 
map $E(x,y)([f])=f(x)-f(y)$. 
It follows directly from the definitions
that
$\|E(x,y)\|\le d(x,y)$. \\
For a given basepoint $o\in M$
the map
$$A_o:M\to \cA^*(M),\ \ \ \ \ \ \ \ x\mapsto E(x,o).$$
Note that $A_o$ is $1$-Lipschitz since
$$\|A_o(x)-A_o(y)\|=\|E(x,y)\|\le d(x,y).$$
If $x,y\in M$ and $z\in m(x,y)$ then for all $[f] \in \cA$ we have
$$E(z,o)[f]=f(z)-f(o)=\frac{1}{2}(f(x)+f(y))-f(o)
=(\frac{1}{2}E(x,o)+\frac{1}{2}E(y,o))[f].$$
This implies
\begin{equation} \label{eq:add}
A_o(m)=\frac{1}{2}(A_o(x)+A_o(y))
\end{equation}
We now show that $A_o$ is actually an isometric map.
Let $x,y \in M$, then we have a geodesic minus interval
$c_{xy}:\R\sm I \to M$. Recall that $c_{xy}(\R\sm I)$ contains at least one ray
and let $b$ the Busemann function of such a ray.
Then $b$ is $1$-Lipschitz and $|b(x)-b(y)| =d(x,y)$. \\
Hence 
$$\|A_o(x)-A_o(y)\| = \| E(x,y) \| \ge |E(x,y)([b])|=|b(y)-b(x)|=d(x,y).$$
Thus $A_o$ is an isometric embedding of $M$ into
the Banachspace $\cA^*=\cA^*(M)$.
The equation (\ref{eq:add}) shows that
the geodesic minus interval is mapped to an
affine line minus interval in the Banachspace $\cA^*$. \\
For simplicity we consider in the sequel
$M$ as a subset of the vector space $\cA^*$.

\begin{lemma} \label{l:1}
Let $x_0,\ldots,x_n\in M\sub \cA^*$ be finitely many points.
Let $E$ be the affine hull of these points, i.e. the smallest affine subspace containing these points.
Then $E\cap M$ contains an open set in the induced topology of $E$.
\end{lemma}

\begin{proof}

This is proven by induction on the dimension of $E$, where the case of
dimension $0$ is trivial.
Assume now that the affine hull of $S=\{x_0,\ldots,x_n\}$ has dimension $(m+1)$.
Then there exists a subset $S'\sub S$, that the dimension of the affine hull
$E'$ of $S'$ is $m$ and by induction hypothesis there exists an open subset
$U'\sub E'$, with $U'\sub M$.
Furthermore there exists a point $x\in S$ such that
$x\in E\sm E'$.
For every $u\in U'$ consider the
line minus interval $c_{xu}:\R\sm I_u\to M\cap E$.
We reparameterize these lines minus interval proportionally to arclength such that 
$c_{xu}(0)=x$ and $c_{xu}(1)=u$.
Then $I_u\sub R\sm\{0,1\}$ is an open interval. If for all $u\in U'$,
$I_u\cap [0,1] =\emptyset$, then $\bigcup_{u} c_{xu}((0,1)) \sub M$ is an open subset of $E$.
If for some $u_0$ we have $I_{u_0}\cap [0,1]\ne \emptyset$ then
$I_{u_0}\sub (0,1)$ and then also $I_u\sub (0,1)$ for all $u\in U''$, where
$U''$ is an open neighborhood of $u_0$ in $U'$. Then
$\bigcup_{u\in U''}c_{x,u}((1,\infty))\sub M$
is an open subset of $E$.

\end{proof}

\begin{lemma} \label{l:2}
For all finite subsets $S\sub M$, the affine hull of $S$ in $\cA^*$ is a Euclidean subspace.
\end{lemma}

\begin{proof}
$E\cap M$ contains an open subset of $E$ by Lemma \ref{l:1}, and $E\cup M$ is a Ptolemy space.
Thus, by Corollary \ref{cor:sch} of Schoenberg's result, $E$ is a Euclidean space.
\end{proof}

\begin{lemma} \label{l:3}
Let $E$ be an affine Euclidean subspace of $\cA^*$,
then $(E\cap M)\cup \{\infty\}$ is a $3$-point Ptolemy segment space.
\end{lemma}

\begin{proof}
Let $x_1,x_2,x_3\in (E\cap M)\cup \{\infty\}$ be distinct points.
By assumption on $M$ the exists a Ptolemy segment $\sigma$ in $M\cup\{\infty\}$
containing these points. We have to show that $\sigma \sub E\cup\{\infty\}$.
We actually show that $\sigma \sub E'\cup\{\infty\}$, where $E'$ is the affine hull
of $\{x_1,x_2,x_3\}$. Let $x_4\in \sigma$, we show that $x_4 \in E'\cup\{\infty\}$.
We can assume $x_4\ne \infty$ and furthermore that all four points
are distinct.
Take the affine hull $E''$ of $\{x_1,x_2,x_3,x_4\}$, then $E'\sub E''$.
By Lemma \ref{l:2} $E''$ is a euclidean space and the four points satisfy the Ptolemy
equality. By the classical Ptolemy theorem this implies that $x_4$ is contained
in the affine hull of the three other points.
\end{proof}

\begin{lemma} \label{l:4}
The affine hull of $M$ is contained in a finite dimensional Euclidean space.
\end{lemma}

\begin{proof}
It suffices to show that there is a number $k \in \N$, such that
the dimension of the affine hull of any finite subset
$S \sub M$ is bounded by $k$.
If we assume the contrary, there are points $x_0,x_1,\ldots \in M$, such
that the affine hull of
$x_0,\ldots,x_n$ is an $n$-dimensional Euclidean affine subspace $E^n \sub \cA^*$.
By Lemma \ref{l:3}, $M\cap E^n$ is a $3$-point Ptolemy segment space, and hence
$M\cap E^n$ is either $E^n$, isometric to a closed halfspace
or isometric to the complement of a distance ball in $E^n$ by Proposition
\ref{prop:eucl}.
In every case $E\cap M$ contains a complete hyperplane $P^{n-1}$ of dimension $(n-1)$ through the point $x_0$. We can arrange the hyperplanes such that $P^n\sub P^{n+1}$
for all $n$. Since $M$ is locally compact, this is a contradiction.
\end{proof}

Now we can finish the proof of Theorem \ref{theo-3pt-int}.
By Lemma \ref{l:4}, $M$ can be considered as a subset of the classical space
$\R^n\cup \{\infty\}$ and for this case the result follows from Proposition \ref{prop:eucl}. \\

Now we give the

\begin{proof} (of Proposition \ref{pro:Busemann_convex_segment})

Let $c:[0,\infty )\longrightarrow X\setminus \{ z\}$ be an arbitrary geodesic ray in  $(X\setminus \{ z\} , d_z)$ and let $b$ a Busemann function of that ray.
Consider again  $x,y,m \in X$ with
$a:=d(x,y)/2=d(x,m)=d(m,y)$.
As in the proof of \ref{pro:Busemann_convex_circle}
we have again to show the equation (\ref{eq:0}):

\begin{equation} 
b(m)\ge \frac{1}{2}(b(x)+b(y))
\end{equation}

Consider $w_i=c(i)$ for $i\to \infty$. By assumption there exists a
Ptolemy segment $\sigma_i$ containing $x,y,w_i$.

We restrict the segment that it contains two of these points as 
boundary points and one as an interior point.
After choice of a subsequence, we can assume that one
of the following three cases occurs: all $\sigma_i$ have

(i) $y$ as interior point, (i') $x$ as interior point ,(ii) $w_i$ as interior point.

The cases (i) and (i') are symmetric, so we only consider (i) and (ii).

Now the case (i) is exactly the case which we considered in the proof
of Proposition \ref{pro:Busemann_convex_circle}. 
It remains to consider the second case.

In case (ii)
there exists a PT segment $\sigma_i$ with boundary points $x$ and $y$ and interior
point $w_i$. Here we choose the orientation, such that
$y<w_i<x$ and again we choose $u_i\in \sigma_i$, with
$y<u_i<w_i<x$ such that $|yu_i|=a$.
The PT equality for this segment implies again
$$3a\,d(y,w_i)\ge d(y,w_i)\, d(x,u_i) = 2a\, d(u_i,w_i)+\, a\, d(x,w_i)$$
and hence we again obtain
estimate (\ref{eq:1}).

Again for a subsequence $u_i\to u$ and we obtain
(\ref{eq:2}).

The PT equality for $\sigma_i$ says
$$d(x,u_i)\,d(w_i,y)=d(x,y)\,d(w_i,u_i)\,+\,d(y,u_i)\, d(w_i,x),$$
this implies again for $i\to \infty$, that
$d(x,u)=d(x,y)+d(y,u)=3a$,
which again implies that $y$ is a midpoint of
$m$ and $u$ and hence we have (\ref{eq:3}) again.

In the same way as above we obtain (\ref{eq:0}).

\end{proof}

Finally we give the

\begin{proof} (of Proposition \ref{prop:eucl})
Consider
$X\sub \R^k\cup\{\infty\}$
as in the assumption.
Let $H$ be the smallest affine subspace of $\R^k$ containing
$X\setminus \{\infty\}$.
To simplify the notation we assume that
$H=\R^k$ (otherwise the result follows by induction).

Case 1: $X\setminus\{\infty\}= \R^k$, then we are finished.

\vspace{0.5cm}

Case 2: Assume that $X\setminus\{\infty\}$ is contained in some halfspace.
By changing coordinates we may assume 
$X\setminus\{\infty\} \sub \R^{k-1}\times [0,\infty)$.
Let $x,y \in X\setminus\{\infty\}$ and assume that for the last
coordinates we have
$x_k <y_k$. Since
$x,y,\infty$ are on a circle-segment (i.e. line-segment) contained in $X$,
but $X\setminus\{\infty\}$ does not intersect the lower halfspace,
we see that the semiline
$x+t(y-x)$, $t\ge 0$ is completely contained in $X$.

It is elementary (but somewhat combersome) to prove the following fact:

\vspace{0.3cm}
Let $B\sub \R^{k-1}\times [0,\infty)$ be a closed subset, with the following properties:

(1) $B$ contains $(k+1)$ affinely independent points.

(2) if $x,y \in B$ with $x_k < y_k$ , then also $x+t(y-x) \in B$ for $t\ge 0$.

\noindent Then $B=\R^{k-1}\times [a,\infty)$ for some $a\ge 0$.

This the proves the second case.

\vspace{0.3cm}

Case 3: In the remaining case we consider
that $X\setminus\{\infty\}$ is not contained in a halfspace
and not the complete $\R^k$.
Then there exists an open distance ball $D\subset \R^k \setminus X$,
such that there exists a point $x\in \partial D \cap X$.
By changing coordinates we can assume that $x=0$ is the origin.
Now apply the involution
$x\mapsto \frac{x}{\|x\|}$,
$0\mapsto \infty$, $\infty \mapsto 0$,
which is a M\"obius map and maps circles to circles.
Under this involution $D$ goes to an open halfspace and
hence we have reduced this case to case 2.

\end{proof}
%%%%%%%%%%%%%%%%%%%%%%%%%%%%%%%%%%%%%%%%%%%%%%%%%%%%%%%%

\bigskip
\begin{tabbing}

Thomas Foertsch,\hskip10em\relax \= Viktor Schroeder,\\ 

Mathematisches Institut,\>
Institut f\"ur Mathematik, \\

Universit\"at Bonn,\> Universit\"at Z\"urich,\\
Endenicher Allee 60, \>
 Winterthurer Strasse 190, \\

D-53115 Bonn, Deutschland\>  CH-8057 Z\"urich, Switzerland\\

{\tt foertsch@math.uni-bonn.de}\> {\tt vschroed@math.uzh.ch}\\

\end{tabbing}

\end{document}